\documentclass[11pt,a4paper,final,reqno]{amsart}

\newcommand{\bb}[1]{\mathbb{#1}}
\newcommand{\cl}[1]{\mathcal{#1}}

\newtheorem{theorem}{Theorem}[section]

\newtheorem{lemma}[theorem]{Lemma}
\newtheorem{prop}[theorem]{Proposition}

\newtheorem{thm}[theorem]{Theorem}
\newtheorem{rmk}[theorem]{Remark}

\usepackage{amsmath}
\usepackage{amssymb}
\usepackage{fullpage}
\usepackage{graphicx}
\usepackage{hyperref}
\usepackage[makeroom]{cancel}

\newcommand{\mbbZ}{\mathbb{Z}}

\newcommand{\mbbQ}{\mathbb{Q}}

\newcommand{\mbbC}{\mathbb{C}}
\newcommand{\mbbI}{\mathbb{I}}

\newcommand{\mfS}{\mathfrak{S}}
\newcommand{\mfA}{\mathfrak{A}}
\newcommand{\mcC}{\mathcal{C}}
\newcommand{\mcD}{\mathcal{D}}
\newcommand{\Rep}{\mathrm{Rep}}
\newcommand{\Hom}{\mathrm{Hom}}
\newcommand{\Aut}{\mathrm{Aut}}
\newcommand{\br}{\mathrm{br}}
\newcommand{\trc}{\mathrm{Tr}_{\mcC}}
\newcommand{\Irr}{\mathrm{Irr}}
\newcommand{\Stab}{\mathrm{Stab}}
\renewcommand{\Vec}{\mathrm{Vec}}
\newcommand{\sVec}{\mathrm{sVec}}
\newcommand{\Gal}{\mathrm{Gal}}
\newcommand{\diag}{\mathrm{diag}}

\usepackage{xcolor}
\newcommand{\lemmaref}[1]{Lemma \ref{#1}}
\newcommand{\propref}[1]{Proposition \ref{#1}}
\newcommand{\thmref}[1]{Theorem \ref{#1}}
\newcommand{\sectionref}[1]{Section \ref{#1}}

\newcommand{\mbfD}{\mathbf{D}}
\renewcommand{\th}{\theta}
\renewcommand{\a}{\alpha}

\newcommand{\s}{\sigma}
\newcommand{\e}{\epsilon}
\renewcommand{\r}{\rho}
\newcommand{\paren}[1]{\left(#1\right)}
\newcommand{\lcb}{\left\{}
\newcommand{\rcb}{\right\}}
\newcommand{\abs}[1]{\left|#1\right|}
\newcommand{\Fib}{\mathrm{Fib}}
\newcommand{\ndiv}{\nmid}
\DeclareMathOperator{\SO}{SO}

\newcommand{\hide}[1]{}
\newcommand{\pjbcomment}[1]{\color{red}#1 -pjb\color{black}}
\newcommand{\cmocomment}[1]{\color{blue}#1 -cmo\color{black}}

\begin{document}

\title[Classification of Rank 5 Premodular Categories]{Classification of Rank 5 Premodular Categories}

\author{Paul J. Bruillard}
\email{paul.bruillard@pnnl.gov}

\author{Carlos M. Ortiz}
\address{Pacific Northwest National Laboratory, Richland, WA}
\email{carlos.ortiz@pnnl.gov}

\thanks{\textit{PNNL Information Release:} PNNL-SA-120942}
\thanks{The authors would like to thank C\'{e}sar Galindo and Zhenghan Wang for enlightening
discussions. The research described in this paper was, in part, conducted under
the Laboratory Directed Research and Development Program at PNNL, a
multi-program national laboratory operated by Battelle for the U.S. Department
of Energy.} 

\begin{abstract}
  We survey a number of classification tools developed in recent years and
  employ them to classify pseudo-unitary rank $5$ premodular categories up to Grothendieck
  equivalence.
\end{abstract}

\maketitle




\section{Introduction}

Fusion categories axiomatize and generalize the theory of representation theory,
and their study encompasses not only the representation of finite groups and Lie
groups, but also Hopf algebras. In many of these situations the fusion category
associated to the underlying group/algebra enjoys extra structure, such as a
notion of commutativity (\textit{braiding}), duality (\textit{rigidity}),
non-degeneracy (\textit{modularity}), or other compatibility conditions such as
a \textit{spherical structure} \cite{BK}. The study of fusion categories has moved beyond
its roots in groups and algebras and now have more widespread use. For instance,
nondegenerate ribbon braided fusion categories, \textit{i.e.}, modular
categories, have broad uses in physics where they describe topological phases of
matter, and topological quantum computers\cite{K1,Wang}. Modular categories also have
applications in pure mathematics providing knot, link, and 3-manifold invariants
through TQFT \cite{Tu1}. More recently, ribbon braided fusion categories, \textit{i.e.},
premodular categories, have garnered increased attention. These categories are
thought to describe higher dimensional TQFT and thus have relevance in manifold
invariants and physics \cite{WW}. Furthermore, the study of premodular and modular
categories often follow a ``leap-frogging'' pattern whereby advances in the
understanding of one type of category allow for advances in another. In recent
years, researchers have found it useful to stratify fusion categories by a
numeric parameter known as their \textit{rank}. The low-rank classification of premodular
categories lag the classification of modular categories. On the premodular side
categories are understood through rank 4 \cite{O1,O2,B1}, while on the modular side they are
completely characterized through rank 5 and partially characterized through rank
11 \cite{RSW1,BGNPRW1,BR1,BNRW1}. The lack of advancement in the classification of premodular categories is
slowing the classification of modular categories. In this work we aim to utilize
a technique known as \textit{de-equivariantization} to produce modular
categories from premodular categories. This will allow us to leverage recent
advances in arithmetic properties of modular categories, \textit{e.g.} 
\cite{BNRW2}, to classify pseudo-unitary premodular categories of rank 5 up to Grothendieck
equivalence. Specifically, we will show:
\begin{thm}
	If $\mcC$ is a pseudo-unitary rank 5 premodular category, then
	\begin{itemize}
		\item $\mcC$ is symmetric and is given by $\Rep(G,z)$ where $G$ is
    $\mbbZ_5$, $D_8$, $Q_8$, $D_{14}$, $\mbbZ_5\rtimes \mbbZ_4$, $\mbbZ_7
    \rtimes \mbbZ_3$, $\mfS_4$, or $\mfA_5$, and $z$ is a central element of
    order at most 2.
		\item $\mcC$ is properly premodular and Grothendieck equivalent to:
    \begin{itemize}
      \item $\Rep(D_{14})$ with $\mcC'\cong\Rep\paren{\mbbZ_{2}}$, and $d_{i}=\paren{1,1,2,2,2}$.
      \item $\Rep(\mfS_4)$ with $\mcC'\cong\Rep\paren{\mfS_3}$,
      $d_{i}=\paren{1,1,2,3,3}$, and $T=\diag\paren{1,1,1,-1,-1}$.
      \item $\Rep(D_8)$ with $\mcC'\cong\Rep\paren{\mbbZ_2}$,
      $d_{i}=\paren{1,1,2,1,1}$, and 
      $T=\diag\paren{1,1,\th,-1,-1}$, where $\th$ is a root of unity satisfying
      a monic degree 4 polynomial over $\mbbZ$.
      \item $PSU(2)_8$, and is obtainable as a $\bb{Z}_2$-equivariantization of $Fib\boxtimes Fib$.
    \end{itemize}
		\item $\mcC$ is modular and it is Grothendieck equivalent to $SU(2)_4, SU(2)_9/\mbbZ_2, SU(5)_1$, or $SU(3)_4/\mbbZ_3$ \cite{BNRW2}.
	\end{itemize}
  Moreover, each case is realized.
\end{thm}

In \sectionref{Section: Preliminaries} we will review the basic theory of premodular
categories. Having dispensed with the preliminaries we will stratify premodular
categories by the amount of degeneracy. In \sectionref{Section: Classification} we
will analyze each case in turn to arrive at \thmref{Theorem: Main Theorem}.

\section{Preliminaries}
\label{Section: Preliminaries}


A premodular category $\cl C$ is a braided, balanced, and fusion category. We will denote the
isomorphism classes of simple objects by $X_{a}$, indexed such that
$X_{0}=\mbbI$ is the monoidal unit. We will denote the set of such isomorphism
classes by $\Irr(\mcC)$.
The fusion matrices, $N_{a}$, the $S$-matrix, $S=(s_{x,y})$, and the $T$-matrix, $T=(\delta_{x,y}\theta_{x})$
are defined in the usual way \cite{BK}. Here $\theta_{x}$ is the twist of the
simple $x$ and is known to have finite order \cite{V}. The triple
$(\{N_{a}\},S,T)$ is known as \textbf{premodular datum}. Throughout we will
assume $\mcC$ is \textbf{pseudo-unitary} and so we may take the categorical
dimensions, $d_{a}$, are the Frobenius-Perron eigenvalues of the $N_{a}$,
\textit{i.e.}, the FP-dimensions.

Let $\cl C'$ denote the M\"{u}ger center of the category
$\cl C$ \cite{muger}. We say that $\mcC$ is \textbf{modular} if $\mcC'=\Vec$,
\textbf{symmetric} if $\mcC=\mcC'$, and \textbf{properly premodular} otherwise.
Symmetric categories are completely characterized by the following result due to
Deligne:
\begin{thm}[See \cite{D1}]
  \label{Theorem: Deligne Symmetric Classification}
	If $\cl C$ is a symmetric ribbon category, then $\cl C \cong \Rep(G,z)$ where
  $z$ is a central element of $G$ of order at most $2$.
\end{thm}
In the properly premodular setting, this result can be exploited to determine
part of the premodular datum. For instance, the categorical dimensions of
simples in $\mcC'\cong\Rep(G,z)$ follow from the representation theory of $G$. Of
course, by definition of the $S$-matrix, for $x,y$ isomorphism classes of
simples in $\mcC'$, one has $s_{x,y}=d_{x}d_{y}$. One can obtain similar
information regarding the $T$-matrix, for instance:
\begin{lemma} \label{Lemma: Tannakian}
	Suppose $\cl C$ is a symmetric category and $X \in \cl C$ is an element of
  order $2$, then $\theta_X = \pm 1$. Moreover, $\theta_X = 1$ if $X$ fixes an
  element of $\cl C$. In particular, this is true if $\mcC$ is the M\"{u}ger
  center of an odd rank premodular category.
\end{lemma}
\begin{proof}
	The first statement of this lemma is proven in \cite{D1}. Now if
  $\theta_X=-1$ and $X$ is an element of order $2$ we know that $\langle 1,X
  \rangle \cong \sVec$. Now by \cite[Proposition 2.6 (i)]{ENO2} we know that $X$
  moves everything in the category.
\end{proof}

In the case that $\mcC$ has odd rank, this lemma tells us that the M\"{u}ger
center of $\mcC$ is Tannakian. In particular, we can exploit the minimal
modularization of \cite{Brug}. This procedure is more generally known as
de-equivariantization. Recent work of Natal\'{e} and Bruciu allow one
to gain a great deal of insight into the structure of the minimal modularization
of $\mcC$. It is often true that the structure enforced by the
de-equivariantization procedure is at odds with the Galois theory of the
modular category. A complete discussion of de-equivariantization and the Galois
theory of modular categories is beyond the scope of this current work, but
further details can be found in \cite{Brug,BN1,BGNPRW1}. 

Next, recall that the $T$-matrix, $S$-matrix, and dimensions are related through
the \textbf{balancing equation} \cite{BK}:
\begin{equation}\label{balancing}
	\theta_{x} \theta_{y}S_{xy} = \sum_{k=0}^4 N_{x^* y}^k \theta_k d_k
\end{equation}

Furthermore, recall that the $d_{x}$ are further related through the
\textbf{dimension equation}:
\begin{equation}\label{dimension}
	d_{x}d_{y}=\trc(X\otimes Y)=\trc(\bigoplus
	N_{x,y}^{z}Z)=\sum_{z}N_{x,y}^{z}d_{z}
\end{equation}

The structure of the fusion rules, knowledge of dimensions, and twists of a
subcategory often allow one to produce a pair of polynomials from these
equations from which the order of certain twists can be bounded. In particular,
we have:
\begin{lemma}
  \label{Lemma: Cyclotomic Remark}
  If $\mbfD$ is a subset of $\lcb d_{a}\mid 1\leq a\leq r\rcb$,
  $f\in\mbbQ\paren{\mbfD}[x]$ is a degree $2$ polynomial with leading coefficient
  $a_{2}$ and constant coefficient $a_{0}$, such that $f\paren{\th}=0$ for some
  $\theta$, then 
  \begin{align*}
    [\mbbQ\paren{\th}:\mbbQ]\leq \begin{cases}
    2^{\abs{\mbfD}+1} & \text{if $a_{2}\neq 0$ and $a_{0}/a_{2}$ is a unit}\\ 
    2^{\abs{\mbfD}} & \text{if $a_{2}=0$ or $a_{0}/a_{2}$ is not a unit.}
    \end{cases}
  \end{align*}
\end{lemma}
\begin{proof}
  Let $\a$ be a primitive element of $\mbbQ\paren{\mbfD}$  Then the minimal
  polynomial of $\theta$ over $\mbbQ\paren{\mbfD}$ divides $f$ in
  $\mbbQ\paren{\mbfD}$. In particular, if we let $m$ denote the degree of
  $\mbbQ\paren{\mbfD,\th}$ over $\mbbQ\paren{\mbfD}$, then $m\leq 2$. 
  Note that the minimal polynomial of $\th$ over $\mbbQ\paren{\mbfD}$ must
  divide the minimal polynomial of $\th$ over $\mbbQ$. Since $\th$ is a root of
  unity we can conclude that the minimal polynomial over $\mbbQ\paren{\mbfD}$
  can only possibly have degree $2$ if $a_{2}\neq0$ and $a_{0}/a_{2}$ is a unit.
  In particular, 
  \begin{align*}
    m\leq \begin{cases}
      2 & \text{if $a_{2}\neq 0$ and $a_{0}/a_{2}$ is a unit}\\
      1&\text{if $a_{2}=0$ or $a_{0}/a_{2}$ is a unit.}
    \end{cases}
  \end{align*}

  Next note that for any $d\in\mbfD$, $d$ satisfies a degree $2$ monic polynomial
  over $\mbbQ\paren{\mbfD\setminus\lcb d\rcb}$. In particular,
  $\mbbQ\paren{\mbfD}$ has degree at most $2^{\abs{\mbfD}}$ over $\mbbQ$. Thus
  $\mbbQ\paren{\mbfD,\th}$ has degree at most $m2^{\abs{\mbfD}}$ over $\mbbQ$.
  Since $\mbbQ\paren{\th}$ is a subfield the result follows.
\end{proof}
\begin{lemma}
  \label{Lemma: Last S-matrix Entry}
  Let $\mcC$ be a rank $r$ premodular category with rank $r-1$ M\"{u}ger center,
  and order the simples so that the last object, $X_{r}$, is not in $\mcC'$.
  Then $S_{r,r}=-\dim\mcC'$.
\end{lemma}
\begin{proof}
  Since $X_{r}$ is not in the M\"{u}ger center, the $r$-th column of the
  $S$-matrix must be orthogonal to the $0$-th column. However, $S_{\ell,0}=d_{\ell}$
  for $1\leq \ell\leq r$, and $S_{k,r}=d_{k}d_{r}$ for $1\leq k\leq r-1$. Thus
  the orthogonality condition reads $d_{r}\dim\mcC=d_{r}S_{r,r}$.
\end{proof}

\section{Classification of Rank 5 Premodular Categories}
\label{Section: Classification}
In this section we will classify rank 5 pseudo-unitary premodular categories
upto premodular datum, and in particular Grothendieck equivalence. The
classification of pseudo-unitary rank 5 modular categories can be found in
\cite{BNRW2} so it suffices to consider symmetric and properly
premodular categories. 

We begin with symmetric premodular categories. By 
\thmref{Theorem: Deligne Symmetric Classification}  and
\lemmaref{Lemma: Tannakian} it suffices to determine all groups which have exactly $5$
irreducible representations. The number of irreducible representations can be
related to the order of the group via a classical number theoretic argument
\cite[lemma 4.4 (ii)]{BR1}. In particular, we have
\begin{equation*}
  5 \leq |G| \leq A_5 \Longrightarrow 5 \leq |G|\leq 1806
\end{equation*}
Applying these bounds one can perform an exhaustive search in GAP to deduce:
\begin{prop}
  \label{Prop: Symmetric}
	If $\mcC$ is a symmetric rank 5 category, then $\mcC$ is Tannakian and is
  given by $\Rep(G)$ where 
  $G$ is $\mbbZ_5$, $D_8$, $Q_8$, $D_{14}$, $\mbbZ_5\rtimes \mbbZ_4$, $\mbbZ_7
  \rtimes \mbbZ_3$, $\mfS_4$, or $\mfA_5$.
\end{prop}
Having completed the classification of rank 5 symmetric categories we find it
convenient to stratify our analysis of rank 5 properly premodular categories by the
rank of the M\"{u}ger center. From Propositions \ref{Prop: Symmetric}, \ref{Prop: Rank 4 Muger Center},
\ref{Prop: Rank 3 Muger Center}, and \ref{Prop: Rank 2 Muger Center} we will be
able to prove:
\begin{thm}
  \label{Theorem: Main Theorem}
	If $\mcC$ is a pseudo-unitary rank 5 premodular category, then
	\begin{itemize}
		\item $\mcC$ is symmetric and is given by $\Rep(G,z)$ where $G$ is
    $\mbbZ_5$, $D_8$, $Q_8$, $D_{14}$, $\mbbZ_5\rtimes \mbbZ_4$, $\mbbZ_7
    \rtimes \mbbZ_3$, $\mfS_4$, or $\mfA_5$, and $z$ is a central element of
    order at most 2.
		\item $\mcC$ is properly premodular and Grothendieck equivalent to:
    \begin{itemize}
      \item $\Rep(D_8)$ with $\mcC'\cong\Rep\paren{\mbbZ_2}$,
      $d_{i}=\paren{1,1,2,1,1}$, and 
      $T=\diag\paren{1,1,\th,-1,-1}$, where $\th$ is a root of unity satisfying
      a monic degree 4 polynomial over $\mbbZ$.
      \item $\Rep(D_{14})$ with $\mcC'\cong\Rep\paren{\mbbZ_{2}}$, and
      $d_{i}=\paren{1,1,2,2,2}$.
      \item $\Rep(\mfS_4)$ with $\cl C'\cong\Rep\paren{\mfS_3}$,
      $d_{i}=\paren{1,1,2,3,3}$, and $T=\diag\paren{1,1,1,-1,-1}$.
      \item $PSU(2)_8$, and is obtainable as a $\bb{Z}_2$-equivariantization of $Fib\boxtimes Fib$.
  
    \end{itemize}
		\item $\mcC$ is modular and it is Grothendieck equivalent to $SU(2)_4, SU(2)_9/\mbbZ_2, SU(5)_1$, or $SU(3)_4/\mbbZ_3$ \cite{BNRW2}.
	\end{itemize}
  Moreover, each case is realized.
\end{thm}

\subsection{Rank 4 M\"uger Center}

\begin{prop}
  \label{Prop: Rank 4 Muger Center}
	There is no rank $5$ properly premodular category with a rank $4$ M\"uger center.
\end{prop}
In order to prove this result we proceed through Propositions \ref{Prop: Muger
Center Z4 or Klein 4}, \ref{Prop: Muger Center D10}, \ref{Prop: Muger Center
A4}.  These propositions consider the different structures of $\mcC'$.

By \thmref{Theorem: Deligne Symmetric Classification}, the structure of $\mcC'$
is $\Rep(G,z)$ as a symmetric category. Moreover, by \cite[Lemma 4.4(ii)]{BR1}
we have $|G|\leq 42$. Using GAP we can conclude that $G$ is isomorphic to
$\mbbZ_4$, $\mbbZ_2\times \mbbZ_2$, $D_{10}$, or $\mfA_4$. In each of these
cases we may apply \lemmaref{Lemma: Last S-matrix Entry} to deduce that
$s_{4,4}=-4$. Furthermore, in each case we can apply \thmref{Theorem: Deligne
Symmetric Classification} and \lemmaref{Lemma: Tannakian} along with
\eqref{balancing} applied to $S_{a,a}$ for $0\leq a\leq 4$ to deduce that
$\th_{a}=1$ for $0\leq a\leq 3$.  In particular, the $S$- and $T$-matrices are
determined by the dimension, $d$, and twist, $\th$, of $X_{4}$. Moreover, by
\lemmaref{Lemma: Cyclotomic Remark} we can conclude that the order, $n$, of
$\th$ satisfies $n\in\lcb 1,2,3,4,5,6,8,10,12\rcb$. The cases of $\mbbZ_{4}$ and
$\mbbZ_{2}\times\mbbZ_{2}$ are sufficiently similar that we consider them in
tandem.

\begin{prop}
  \label{Prop: Muger Center Z4 or Klein 4}
  There is no rank 5 premodular category with M\"{u}ger center Grothendieck
  equivalent to $\Rep\paren{\mbbZ_{4}}$ or $\Rep\paren{\mbbZ_{2}^{2}}$.
\end{prop}
\begin{proof}
  In both cases the only $n\in\lcb 1,2,3,4,5,6,8,10,12\rcb$ consistent with
  equations \eqref{balancing} and \eqref{dimension} is $n=10$, consequently
  $d=1+\sqrt{5}$. Letting $G$ denote the group such that
  $\mcC'\cong\Rep\paren{G}$ we can de-equivariantize to produce a modular
  category $\mcC_{G}$. Let $\hat{D}$ denote the set of orbits of isomorphism
  classes of the $G$-action and let $D$ denote a transversal of $\hat{D}$.
  Then the simples of $\mcC$ can be understood via equivariantization as pairs
  $\paren{y,V}$ where $y$ runs over $D$ and $V$ ranges over the simple
  projective representations of $\Stab_{G}\paren{y}$. Since $\mbbI$ is
  stabilized by $G$ and $G$ is an abelian group, four of the simples in $\mcC$
  are of the form $\paren{\mbbI,\chi_{a}}$ where $\chi_{a}$ is a irreducible
  character of $G$. The remaining simple is $\paren{y,V}$ where $V$ is the sole
  simple projective representation of $\Stab_{G}\paren{y}$. In the case of
  $G\cong\mbbZ_{4}$ we recall that $H^{2}\paren{\mbbZ_{2k},\mbbC^{\times}}=0$ for
  all $k$, and so we may take $V$ to be the sole irreducible linear
  representation. In particular,
  $\dim\paren{y}=\paren{1+\sqrt{5}}/4\notin\bar{\mbbZ}$, an impossibility. So it
  suffices to consider $G\cong\mbbZ_{2}^{2}$. In this case
  $H^{2}\paren{\mbbZ_{2}^{2},\mbbC^{\times}}\cong\mbbZ_{2}$ and so, $V$ must be
  the 2-dimensional simple projective corresponding to the
  nontrivial 2-cocycle in $H^{2}\paren{\mbbZ_{2}^{2},\mbbC^{\times}}$. In
  particular, $\mcC_{\mbbZ_{2}^{2}}\cong\Fib$. This is not possible
  \cite{GalindoCommunication}. 
  \hide{\cmocomment{Does $\Fib$ admit a Klein-Four action?}}
\end{proof}

\begin{prop}
  \label{Prop: Muger Center D10}
  There is no rank $5$ premodular category $\cl C$ with $\cl C' = Rep(D_{10})$. 
\end{prop}
\begin{proof}
  The only choice for $n\in\lcb 1,2,3,4,5,6,8,10,12\rcb$ consistent with equations
  \eqref{balancing} and \eqref{dimension} are $n=4$ or $12$. In both cases we have
  $d=3$ and $\dim\mcC=19$. This is not possible as
  $\dim\mcC'=10\ndiv\dim\mcC=19$.
\end{proof}
\begin{prop}
  \label{Prop: Muger Center A4}
  There is no rank $5$ premodular category $\cl C$ with $\cl C' = Rep(A_{4})$. 
\end{prop}
\begin{proof}
  The only choice for $n\in\lcb 1,2,3,4,5,6,8,10,12\rcb$ consistent with equations
  \eqref{balancing} and \eqref{dimension} are $n=4$ or $12$. In both cases we have
  $d=2\sqrt{3}$. Next we observe that $\mbbZ_{3}\subset\mfA_{4}$ and so $\mcC_{\mfA_{4}}$ is a
  premodular category with $\mbbZ_{3}$ action. Letting $D$ denote a transversal
  of the orbits of the simple object isomorphism classes under this action we
  can conclude that the objects in $\mcC$ are pairs $\paren{y,V}$ where $y\in D$
  and $V$ is a simple projective representation of $\Stab_{\mbbZ_{3}}\paren{y}$.
  Since $y=\mbbI$ is fixed by the $\mbbZ_{3}$-action we have can deduce that
  three of the simples are $\paren{\mbbI,\chi_{a}}$ where the $\chi_{a}$ are the
  irreducible characters of $\mbbZ_{3}$. The remaining two objects we denote by
  $\paren{y,V}$ and $\paren{z,W}$. Since $\mbbZ_{3}$ has no proper subgroups it
  must be that $y\neq z$ and
  $\Stab_{\mbbZ_{3}}\paren{y}=\Stab_{\mbbZ_{3}}\paren{z}=1$. Applying the
  orbit-stabilizer theorem we know that $\abs{\mbbZ_{3}.y}=\abs{\mbbZ_{3}.z}=3$.
  Thus $y$ has dimension $2/\sqrt{3}\notin\bar{\mbbZ}$, an impossibility.
\end{proof}


\subsection{Rank 3 M\"{u}ger Center}
\begin{prop}
  \label{Prop: Rank 3 Muger Center}
	If $\mcC$ is a rank 5 properly premodular category with a rank $3$ M\"uger
  center, then $\cl C$ is Grothendieck equivalent to $\Rep\paren{\mfS_4}$,
  $\mcC'\cong\Rep\paren{\mfS_{3}}$, and $T=\diag\paren{1,1,1,-1,-1}$.
\end{prop}

We will consider the different possibilities for $\mcC'$ in Propositions \ref{Prop: Muger
Center Z3} and \ref{Prop: Muger Center S3}. However, before doing this we find
it useful to establish the following lemma.
\begin{lemma} \label{galoislemma}
	Suppose $\cl C$ is a modular category with $d_a=d_b\neq 1$ for
	$a\neq b$ and there exists $\sigma\in \Gal\paren{\cl C}$ such that
  $\sigma\paren{a}=0$ , then $d_{\sigma\paren{b}}=\pm 1$. In particular, $\cl C_{pt}
  \ncong \Vec$.
\end{lemma}
\begin{proof}
  The result follows by applying \cite[Theorem 2.7]{RSW1}, and in particular
  $d_{\s\paren{0}}\s\paren{d_{k}}=d_{\s\paren{k}}\e_{\s\paren{k},\s}$, to
  $\s\paren{d_{a}}=\s\paren{d_{b}}$.
  %
\end{proof}

Next we apply \cite[Lemma 4.4(ii)]{BR1} to conclude that groups with $3$
irreducible representations must satisfy 
$3 \leq |G| \leq A_3 \Longrightarrow 3 \leq |G|\leq 6$. In particular,
$G\cong\mfS_{3}$ or $\mbbZ_{3}$.
%

\begin{prop}
  \label{Prop: Muger Center Z3}
  There is no rank 5 properly premodular category whose M\"{u}ger center is
  Grothendieck equivalent to $\Rep\paren{\mbbZ_{3}}$.
\end{prop}
\begin{proof}
  Suppose $\mcC$ is such a properly premodular category. Then $\mcC'$ is
  Tannakian since $\mbbZ_{3}$ has no order 2 central elements. Thus
  $\mcC_{\mbbZ_{3}}$ is modular. Now let $\hat{D}$ be the set of orbits of
  isomorphism classes of simples in $\mcC_{\mbbZ_{3}}$ under the $\mbbZ_{3}$
  action, and let $D$ be a transversal of $\hat{D}$. Then the simples in $\mcC$
  are $\paren{\mbbI,\chi_{a}}$, $\paren{y,V}$, and $\paren{z,W}$, where
  $\chi_{a}$ are the irreducible linear characters of $\mbbZ_{3}$, and $V$ and
  $W$ are simple projective representations of $\Stab_{\mbbZ_{3}}\paren{y}$ and
  $\Stab_{\mbbZ_{3}}\paren{z}$ respectively. Since $\mbbZ_{3}$ has no proper
  subgroups and $H^{2}\paren{\mbbZ_{3},\mbbC^{\times}}=0$ we may take $V$ and
  $W$ to be irreducible linear representations. Of course all irreducibles of
  the stabilizer must appear and so $y\neq z$ and the stabilizers are trivial.
  Applying the Orbit-Stabilizer Theorem, we can conclude that
  $\abs{\mbbZ_{3}.y}=\abs{\mbbZ_{3}.z}=3$ and $\mcC_{\mbbZ_{3}}$ is a rank $7$
  modular category of global dimension $1+3a^{2}+3b^{2}$ where $a=d_{3}/3$ and
  $b=d_{4}/3$. So either $a=b=1$ by \cite{BGNPRW1} or we may apply
  \lemmaref{galoislemma} to conclude that $a=1$. Applying the equidimensionality
  of the universal grading of $\mcC_{\mbbZ_{3}}$, \cite{GN2,DGNO1} we can deduce
  that $b=1$ or $2$. The later case cannot occur as every integral modular
  category of rank 7 is pointed \cite[Theorem 5.8]{BGNPRW1}. Thus the global
  dimension of $\mcC$ is $21$. Applying \cite[Theorem 6.3]{EGO1} we can conclude
  that $\mcC$ is Grothendieck equivalent to
  $\Rep\paren{\mbbZ_{3}\rtimes\mbbZ_{7}}$. Applying these fusion rules,
  balancing, and the symmetry $s_{34}=s_{43}$ allows us to conclude that
  $\th_{3}=\th_{4}$. The last column of $S$ is now determined in terms of
  $\th_{3}$ by balancing and the fusion rules. 
%
  Applying orthogonality of the first and last column of the $S$-matrix produces
  the equation $\th_{3}^{2}+5\th_{3}+1=0$. This is not possible as $\th_{3}$ is
  a root of unity.
\end{proof}

\begin{prop}
  \label{Prop: Muger Center S3}
  If $\mcC$ is properly premodular and $\mcC'$ has rank $3$, then
  $\mcC$ is Grothendieck equivalent to $\Rep\paren{\mfS_{4}}$,
  $\mcC'\cong\Rep\paren{\mfS_{3}}$, and $T=\diag\paren{1,1,1,-1,-1}$. Such a
  category can be obtained via the $\mfS_{3}$-equivariantization of a 3-fermion
  theory.
\end{prop}
\begin{proof}
  Just as in the proof of \propref{Prop: Muger Center Z3} we may
  de-equivariantize and understand the simples of $\mcC$ in terms of
  equivariantization. Since $H^{2}\paren{\mfS_{3},\mbbC^{\times}}=0$ we have
  that the simples are $\paren{\mbbI,\r_{a}}$, $\paren{y,V}$, and $\paren{z,W}$
  where $\r_{a}$ are the irreducible representations of $\mfS_{3}$, while $V$
  and $W$ are the irreducible linear representations of
  $\Stab_{\mfS_{3}}\paren{y}$ and $\Stab_{\mfS_{3}}\paren{z}$ respectively.
  There are now two cases to consider:

  \textbf{Case 1: $y=z$} In this case $\Stab_{\mfS_{3}}\paren{y}$ can only have
  two irreducible linear representations and hence is isomorphic to $\mbbZ_{2}$.
  By the Orbit-Stabilizer Theorem we can conclude that $\mcC_{\mfS_{3}}$ has
  rank $4$ and thus is pointed by \cite{RSW1}. 

  By \cite{BN1} we know that the simple objects in $\cl C$ have dimensions $1,
  1, 2, 3, 3$ and the balancing equation easily tell us that
  $\theta_0=\theta_1=\theta_2=1$. In addition, if we de-equivariantize by
  $\Rep(\mbbZ_2)\subset \Rep(\mfS_3)=\cl C'$, we get a premodular category with
  simples with dimensions $1,1,1,3$.  If we apply \cite[Proposition 3.21]{BN1}
  to the object of dimension $3$ we can conclude that $x_3=x_3^*$ and
  $x_4=x_4^*$. In addition, we know that $X_1\otimes X_3=X_4 \implies X_3\otimes
  X_3$ does not contain $X_1$ in its decomposition. If we know look at the
  dimension equation corresponding to $X_3\otimes X_3$ we get the following
  Diophantine equation,
  \begin{equation*}
    9=1+2N_{33}^2+3( N_{33}^3+N_{33}^4) \implies N_{33}^2=1\text{ and }
    N_{33}^3+N_{33}^4=2 \text{ or } N_{33}^2=4\text{ and } N_{33}^3+N_{33}^4=0
  \end{equation*}

  The latter case cannot happen since this would force $4\leq|| N_3 ||_{max}
  \leq d_3=3$ \cite[Lemma 3.14]{BNRW1}.  From here we can finally conclude,
  $N_{23}^3=N_{23}^4=N_{24}^3=N_{24}^4=N_{13}^4=1$ and $N_{33}^3+N_{33}^4 =
  N_{44}^3+N_{44}^4=2$. If we now use the fact that all the $N_i$ matrices
  commute, we can conclude that $N_{33}^4=N_{44}^4$ and $N_{44}^3=N_{33}^3$. If
  we combine the previous mentioned observations and the equation that you get
  by looking at $s_{33}$ we are able to conclude that $\theta_3=\theta_4=-1$.

  Note that if we consider the rank 4 pointed modular categories, $\mcD$, then
  the Grothendieck group is either $\mbbZ_{4}$, or $\mbbZ_{2}^{2}$. In the
  former case there are two $\mfS_{3}$ equivariantizations of rank $9$ and $12$.
  In the later case, if we take $\mcD$ to be equivalent to the toric code, then
  there are again two $\mfS_{3}$ equivariantizations of rank $9$ and $12$. The
  final option is that $\mcD$ is Grothendieck equivalent to the 3-fermion
  theory. In this case $\Aut_{\otimes}^{\br}\paren{\mcD}\cong\mfS_{3}$ and so
  the possible $\mfS_{3}$-actions are given by group homomorphisms,
  $\Hom\paren{\mfS_{3},\mfS_{3}}$. Once again, there are rank $9$ and $12$,
  equivariantizations. However, the identity automorphism yields an
  $\mfS_{3}$-action on $\mcD$ with 2-orbits. The stabilizer of the nontrivial
  orbit is $\mbbZ_{3}$ and so the $\mfS_{3}$-equivariantization under such an
  action would have rank 5 and have simples of dimension $1,1,2,3,3$. Thus
  $\mcC$ can be realized as an $\mfS_{3}$-equivariantization of the 3-fermion
  theory.

  \textbf{Case 2: $y\neq z$} Since all irreducibles must appear we have that the
  stabilizers of $y$ and $z$ in $\mfS_{3}$ are trivial. Applying the
  Orbit-Stabilizer Theorem, we can conclude that $\mcC_{\mfS_{3}}$ is a rank 13
  modular category of dimension $1+6a^{2}+6b^{2}$ where $a=d_{3}/6$ and
  $b=d_{4}/6$. In the case that $\mcC_{\mfS_{3}}$ is integral we may apply the
  techniques of \cite{BR1} and by exhaustive search conclude that $a=b=1$. On
  the other hand, if $\mcC_{\mfS_{3}}$ is not integral in which case we may
  apply \lemmaref{galoislemma} to conclude that $a=1$. Invoking the universal
  grading of \cite{GN2} we have $b=1$ or $b=\sqrt{7}$. In the later case
  $\mcC_{\mfS_{3}}$ has dimension $49$ contradicting \cite[Proposition
  8.32]{ENO1}. Thus $\mcC_{\mfS_{3}}$ is pointed.

  To see that this is not possible, let $\mcD$ be a rank 13 pointed modular
  category. Of course, the orthogonal group on $\mbbZ_{13}$, with quadratic form
  $q$ coming from the twist on $\mcD$, $O\paren{\mbbZ_{13},q}$, is either
  trivial or isomorphic to $\mbbZ_{2}$. So the actions of $\mfS_{3}$ on $\mcD$
  are given by elements of $\Hom\paren{\mfS_{3},O\paren{\mbbZ_{13},q}}$, which
  are either trivial or the sign action. In the case of the trivial action, the
  isomorphism classes of simple objects in $\mcD$ are fixed under the action of
  $\mfS_{3}$.  In particular, the de-equivariantization has rank $39$. In the
  case of the sign action, there are 7 orbits of simples under the $\mfS_{3}$
  action and hence the de-equivariantization has rank $21$. Thus, $\mcC$ is not
  the $\mfS_{3}$-equivariantization of a rank 13 pointed modular category.

\hide{
  By \cite{BN1} we can conclude
  that the simples in $\mcC$ have dimensions $1,1,2,6,6$. 
  Applying the balancing
  equation reveals that $\th_{0}=\th_{1}=\th_{2}=1$. 
  \cmocomment{fixed until this point}
  In addition, if we apply \cite[Proposition 3.12]{BN1} we can conclude that
  $x_3=x_3^*$ and $x_4=x_4^*$. By \cite[Corollary 3.11]{BN1} \pjbcomment{How were
  you able to use this Theorem? We generally don't have enough information and
  corollary 3.11 is more effective.} and the balancing
  equations we can conclude that $\theta_3=\theta_4$,
  $N_{23}^3=N_{23}^4=N_{24}^3=N_{24}^4=N_{13}^4=1$, $N_{33}^3+N_{33}^4=2$, and
  $N_{44}^3+N_{44}^4=2$. If we now use the fact that all the $N_i$ matrices
  commute, we can conclude that $N_{33}^4=N_{44}^4$ and $N_{44}^3=N_{33}^3$. If
  we combining these results with \eqref{balancing} applied to $s_{33}$ gives
  $\th_{3}=\th_{4}=-1$.
}
\end{proof}
\subsection{Rank 2 M\"{u}ger Center}
\begin{prop}
  \label{Prop: Rank 2 Muger Center}
	If $\mcC$ is a rank 5 properly premodular category with a rank 2 M\"uger
  center, then one of the following is true:
  \begin{itemize}
    \item[(i)]  $\mcC$ is Grothendieck equivalent to $\Rep(D_8)$ with
    $C'\cong\Rep(Z_2)$. Furthermore, the $T$-matrix is of the form
    $T=\diag\paren{1,1,\th,-1,-1}$ for some root of unity $\th$, corresponding
    to that $2$-dimensional simple object, satisfying a monic quintic polynomial
    over $\mbbZ$, or
    \item[(ii)] $\mcC$ is Grothendieck equivalent to $\Rep\paren{D_{14}}$.
    \item[(iii)] $\mcC$ is Grothendieck equivalent to $PSU(2)_8$.
  \end{itemize}
  Moreover, all cases are realized. The first case by a
  $\mbbZ_{2}$-equivariantization of the toric code, the second by the adjoint
  subcategory of a $56$-dimensional metaplectic category, \textit{e.g.},
  $\paren{\SO\paren{14}_{2}}_{ad}$, and the final case by a $\bb{Z}_2$-equivariantization of $Fib\boxtimes Fib$.
\end{prop}
\begin{proof}
  Applying \cite[Lemma 4.4(ii)]{BR1} we see that if a group $G$ has only two
  irreducible representations, then $G\cong\mbbZ_{2}$. Thus $\mcC'$ is
  Grothendieck equivalent to $\Rep\paren{\mbbZ_{2}}$. Since $\mcC$ has odd rank
  we can further conclude that $\mcC'$ is Tannakian. Ordering the simples so
  that $X_{0}=\mbbI$ and $X_{1}$ generates $\mcC'$ we have two cases to
  consider.

  \textbf{Case 1:} $X_{1}\otimes X_{3}=X_{4}$ and $X_{1}\otimes X_{2}=X_{2}$.
    It follows immediately from these fusion rules that $X_{2}$ is self-dual.
    Next note that the de-equivariantization, $\mcC_{\mbbZ_{2}}$ is a modular
    category of rank $4$ and two of the objects have the same dimension. By
    \cite{RSW1} we can conclude that $\mcC_{\mbbZ_{2}}$ is pointed, $Fib\boxtimes Fib$, or $Fib \boxtimes Sem$. 

In the pointed case, we get $d_{2}=2$ and
    $d_{3}=d_{4}=1$. Thus $X_{2}\otimes X_{2}=\mbbI\oplus X_{1}\oplus
    X_{1}\oplus X_{2}$ or $\mbbI\oplus X_{1}\oplus X_{3}\oplus X_{4}$. In the
    former case, $X_{2}$ would generate a $6$-dimensional fusion subcategory.
    This is not possible as $6\ndiv 8$. Thus $X_{2}$ generates $\mcC$. Next note
    that this category is necessarily near group and hence can only be braided
    if $X_{3}^{*}=X_{4}$ \cite{S1}, see also \cite[Remark 4.4]{NR1}.
    It follows from \cite[Theorem 4.2]{NR1} that $\mcC$ is Grothendieck
    equivalent to $\Rep\paren{D_{8}}$. Examining the balancing equation for
    $S_{1,3}$ we see that $\th_{3}=\th_{4}$. Since $X_{2}$ is self-dual we can
    conclude that column 2 in $S$ is real. In particular, $\th_{3}=\pm1$.
    Since $X_{3}$ is not in the M\"{u}ger center of $\mcC$ we know that the
    $3$-rd and $0$-th column of $S$ must be orthogonal, hence
    $\th_{3}=\th_{4}=-1$. By \cite[Corollary 3.3]{B1} we know that
    \begin{align}
      \label{equation: theta condition}
      \th^{2}+\th^{-2}=\frac{1}{D^{2}}\sum_{b,c}N_{b,c}^{2}d_{b}d_{c}\paren{\frac{\th_{b}}{\th_{c}}}^{2}\in\mbbZ
    \end{align}

    This yields a quintic polynomial for $\th$ over $\mbbZ$.  Considering all
    possible $n$ such that $\phi\paren{n}<4$ gives possible primitive roots.
    However, such $\th=\zeta_{n}$ with $\phi\paren{n}<4$ satisfy \eqref{equation: theta condition}. Such a category can be constructed by considering the equivariantization of the toric code under the nontrivial $\mbbZ_{2}$-action. 

    In the case where $\mcC_{\mbbZ_{2}}=Fib \boxtimes Fib$, we get the following set of dimensions $\{1,1, \varphi^2, \varphi^2, 2\varphi \}$ for $\cl C$, where $\varphi$ is the golden mean. After computing the fusion rules using \eqref{balancing}, we can conclude that $\cl C$ is Grothendieck equivalent $PSU(2)_8$.

    Finally, if $\mcC_{\mbbZ_{2}}=Fib \boxtimes Sem$, we get a category $\cl C$ with dimensions $\{1,1,1,1,2\varphi\}$. This category violates the condition given in \cite[Theorem 1.1]{S1} for near-group categories, hence it does not exist.

  \textbf{Case 2:} $X_{1}\otimes X_{a}=X_{a}$ for $2\leq a\leq 4$
    In this case the de-equivariantization is rank 7 modular with simples of
    dimension
    $1,\hat{d_{1}},\hat{d_{1}},\hat{d_{2}},\hat{d_{2}},\hat{d_{3}},\hat{d_{3}}$.
    In particular, by \lemmaref{galoislemma} and \cite{BGNPRW1} we know
    that $\paren{\mcC_{\mbbZ_{2}}}_{pt}$ is either $\Rep\paren{\mbbZ_{3}}$ or
    $\mcC_{\mbbZ_{2}}$. In the former case, applying the universal grading and
    the pigeon-hole principle gives $\hat{d}_2^2+\hat{d}_3^2=3$. From here
    notice that the global dimension of this category is $9$ and by
    \cite[Proposition 8.32]{ENO1} we can conclude that
    $\paren{\mcC_{\mbbZ_{2}}}_{pt}=\mcC_{\mbbZ_{2}}\ncong\Rep\paren{\mbbZ_{3}}$,
    a contradiction.

    Thus $\mcC_{\mbbZ_{2}}$ is pointed and hence $d_{1}=d_{2}=d_{3}=2$.
    By the pigeon-hole principle and without loss of generality, we
    may assume $X_{2}$ is self-dual. Just as in case 1 a dimension argument
    shows that $X_{2}$ cannot be a subobject of $X_{2}\otimes X_{2}$. So without
    loss of generality we have $X_{2}\otimes X_{2}=\mbbI\oplus X_{1}\oplus
    X_{3}$. Of course, the category generated  by $X_{2}$ cannot exclude $X_{4}$
    since its dimension must divide $14$. Thus $\mcC$ is cyclically generated by
    $X_{2}$. Furthermore, since $X_{2}$ is self-dual, its tensor square must be
    as well. In particular, $X_{3}$ is self-dual. The pigeon-hole principle
    reveals that $X_{4}$ must also be self-dual. Thus $\mcC$ is Grothendieck
    equivalent to $\Rep\paren{D_{14}}$ by \cite[Theorem 4.2]{NR1}.\footnote{One
    can arrive at the same conclusion by consider equivariantizations of the
    pointed modular categories considered in \cite{RSW1} as in the proof of
    \propref{Prop: Muger Center S3}.} Such categories may be realized as the
    adjoint categories of 56-dimensional metaplectic categories \cite{BPR1}.
\end{proof}

\begin{rmk}
  It is interesting to note that none of the known algebraic conditions,
  \cite{B1}, for premodular categories are sufficient to determine $\th$ in
  \propref{Prop: Rank 2 Muger Center}. Of course, there are finitely many
  modularizable premodular categories and so there must be finitely many such
  $\th$ \cite{BNRW1}.
\end{rmk}

\end{document}